\documentclass{amsart}
\usepackage{amssymb,amsthm, amsmath, amsfonts}
\usepackage{graphics}
\usepackage{hyperref}
\usepackage[all]{xy}
\usepackage{enumerate}
\usepackage[mathscr]{eucal}
\usepackage{bbm}

\theoremstyle{plain}
\newtheorem{theorem}{Theorem}[section]
\newtheorem{lemma}[theorem]{Lemma}
\newtheorem{proposition}[theorem]{Proposition}

\newtheorem{corollary}[theorem]{Corollary}
\numberwithin{equation}{section}

\theoremstyle{definition}

\newtheorem{definition}[theorem]{Definition}

\newtheorem{remark}[theorem]{Remark}

\DeclareMathOperator{\Mod}{-Mod}
\DeclareMathOperator{\module}{-mod}
\DeclareMathOperator{\Hom}{Hom}
\DeclareMathOperator{\Tor}{Tor}
\DeclareMathOperator{\hd}{hd}
\DeclareMathOperator{\gd}{gd}
\DeclareMathOperator{\pd}{pd}
\DeclareMathOperator{\td}{td}
\DeclareMathOperator{\fdim}{findim}
\DeclareMathOperator{\gldim}{gldim}
\DeclareMathOperator{\coker}{coker}

\DeclareMathOperator{\im}{im}

\newcommand{\mk}{{\mathbbm{k}}}
\newcommand{\C}{{\mathscr{C}}}
\newcommand{\tC}{{\underline{\mathscr{C}}}}
\newcommand{\Z}{{\mathbb{Z}_+}}
\newcommand{\FI}{{\mathscr{FI}}}

\title{Filtrations and Homological degrees of FI-modules}
\author{Liping Li}
\address{Key Laboratory of Performance Computing and Stochastic Information Processing (Ministry of Education), College of Mathematics and Computer Science, Hunan Normal University, Changsha, Hunan 410081, China.}
\email{lipingli@hunnu.edu.cn}

\author{Nina Yu}
\address{College of Mathematics, Xiamen University; Xiamen, Fujian, 361000, China.}
\email{ninayu@xmu.edu.cn}

\thanks{The first author is supported by the National Natural Science Foundation of China 11541002, the Construct Program of the Key Discipline in Hunan Province, and the Start-Up Funds of Hunan Normal University 830122-0037. Both authors highly appreciate the anonymous referee for carefully checking the manuscript and providing many valuable comments and suggestions.}

\begin{document}

\begin{abstract}
Let $\mk$ be a commutative Noetherian ring. In this paper we consider $\sharp$-filtered modules of the category $\FI$ firstly introduced in \cite{N}. We show that a finitely generated $\FI$-module $V$ is $\sharp$-filtered if and only if its higher homologies all vanish, and if and only if a certain homology vanishes. Using this homological characterization, we characterize finitely generated $\C$-modules $V$ whose projective dimension $\pd(V)$ is finite, and describe an upper bound for $\pd(V)$. Furthermore, we give a new proof for the fact that $V$ induces a finite complex of $\sharp$-filtered modules, and use it as well as a result of Church and Ellenberg in \cite{CE} to obtain another upper bound for homological degrees of $V$.
\end{abstract}

\maketitle

\section{Introduction}

\subsection{Motivation}

The category $\FI$, whose objects are finite sets and morphisms are injections between them, has played a central role in representation stability theory introduced by Church and Farb in \cite{CF}. It has many interesting properties, which were used to prove quite a few stability phenomena observed in \cite{CE, CEF, CF, Farb, P}. Among these properties, the existence of a \emph{shift functor} $\Sigma$ is extremely useful. For instances, it was applied to show the locally Noetherian property of $\FI$ over any Noetherian ring by Church, Ellenberg, Farb, and Nagpal in \cite{CEFN}, and the Koszulity of $\FI$ over a field of characteristic 0 by Gan and the first author in \cite{GL1}. Recently, Nagpal proved that for an arbitrary finitely generated representation $V$ of $\FI$, when $N$ is large enough, $\Sigma_N V$ has a special filtration, where $\Sigma_N$ is the $N$-th iteration of $\Sigma$; see \cite[Theorem A]{N}. Church and Ellenberg showed that $\FI$-modules have Castlenuovo-Mumford regularity (for a definition in commutative algebra, see \cite{ES}), and gave an upper bound for the regularity; see \cite[Theorem A]{CE}.

The main goal of this paper is to use the shift functor to investigate homological degrees and special filtrations of $\FI$-modules. Specifically, we want to:
\begin{enumerate}
\item obtain a homological characterization of \emph{$\sharp$-filtered modules} firstly introduced by Nagpal in \cite{N}; and
\item use these $\sharp$-filtered modules, which play almost the same role as projective modules for homological calculations, to obtain upper bounds for projective dimensions and homological degrees of finitely generated $\FI$-modules.
\end{enumerate}

In contrast to the combinatorial approach described in \cite{CE}, our methods to realize these objectives are mostly conceptual and homological. We do not rely on any specific combinatorial structure of the category $\FI$. The main technical tools we use are the shift functor $\Sigma$ and its induced cokernel functor $D$ introduced in \cite[Subsection 2.3]{CEFN} and \cite[Section 3]{CE}. Therefore, it is hopeful that our approach, with adaptable modifications, can be applied to other combinatorial categories recently appearing in representation stability theory (\cite{SS}).

\subsection{Notation}

Before describing the main results, we first  introduce necessary notation. Throughout this paper we let $\C$ be a skeletal category of $\FI$, whose objects are $[n] = \{ 1, 2, \ldots, n \}$ for $n \in \Z$, the set of nonnegative integers. By convention, $[0] = \emptyset$. By $\mk$ we mean a commutative Noetherian ring with identity. Given a set $S$, $\underline{S}$ is the free $\mk$-module spanned by elements in $S$. Let $\tC$ be the $\mk$-linearization of $\C$, which can be regarded as both a $\mk$-linear category and a $\mk$-algebra without identity.

A \textit{representation} of $\C$, or a $\C$-\emph{module}, is a covariant functor $V$ from $\C$ to $\mk \Mod$, the category of left $\mk$-modules. Equivalently, a $\C$-module is a $\tC$-module, which by definition is a $\mk$-linear covariant functor from $\tC$ to $\mk \Mod$. It is well known that $\C \Mod$ is an abelian category. Moreover, it has enough projectives. In particular, for $i \in \Z$, the $\mk$-linearization $\tC(i, -)$ of the representable functor $\C(i, -)$ is projective.

A representation $V$ of $\C$ is said to be \emph{finitely generated} if there exists a finite subset $S$ of $V$ such that any submodule containing $S$ coincides with $V$; or equivalently, there exists a surjective homomorphism
\begin{equation*}
\bigoplus _{i \in \Z} \tC(i, -) ^{\oplus a_i} \to V
\end{equation*}
such that $\sum _{i \in \Z} a_i < \infty$. It is said to be \emph{generated in degrees $\leqslant N$} if in the above surjection one can let $a_i = 0$ for all $i > N$. Obviously, $V$ is finitely generated if and only if it is generated in degrees $\leqslant N$ for a certain $N \in \Z$ and the values of $V$ on objects $i \leqslant N$ are finitely generated $\mk$-modules. Since $\C$ is \emph{locally Noetherian} by the fundamental result in \cite{CEFN}, the category $\C \module$ of finitely generated $\C$-modules is abelian. In this paper we only consider finitely generated $\C$-modules over commutative Noetherian rings.

Given a finitely generated $\C$-module $V$ and an object $i \in \Z$, we denote its value on $i$ by $V_i$. For every $n \in \Z$, one can define a \emph{truncation functor} $\tau_n: \C \module \to \C \module$ as follows: For $V \in \C \module$,  $$(\tau_{n}V)_{i}:=\begin{cases}
0, & i<n\\
V_{i}, & i\geqslant n
\end{cases}$$ A finitely generated $\C$-module $V$ is called \emph{torsion} if there exists some $N \in \Z$ such that $\tau_N V = 0$. In other words, $V_i = 0$ for $i \geqslant N$.

The category $\C$ has a \emph{self-embedding functor} $\iota: \C \to \C$ which is faithful and sends an object $i \in \Z$ to $i+1$. For a morphism $\alpha \in \C(i, j)$ (which is an injection from $[i]$ to $[j]$), $\iota (\alpha)$ is an injection from $[i+1]$ to $[j+1]$ defined as follows:
\begin{equation} \label{iota}
(\iota (\alpha)) (r) = \begin{cases}
1, & r = 1 \in [i+1];\\
\alpha(r-1) + 1, & 1 \neq r \in [i+1].
\end{cases}
\end{equation}
The functor $\iota$ induces a pull-back $\iota^{\ast}: \C \module \to \C \module$. The \emph{shift functor} $\Sigma$ is defined to be $\iota^{\ast} \circ \tau_1$. For details, see \cite{CEFN, GL1, Li2}.

By the directed structure, the $\mk$-linear category $\tC$ has a two-sided ideal
\begin{equation*}
J = \bigoplus _{0 \leqslant i < j} \tC(i, j).
\end{equation*}
Therefore,
\begin{equation*}
\tC_0 = \bigoplus _{i \in \Z} \tC(i, i)
\end{equation*}
is a $\tC$-module via identifying it with $\tC / J$.

Given a finitely generated $\C$-module $V$, its \emph{torsion degree} is defined to be
\begin{equation*}
\td(V) = \sup \{ i \in \Z \mid \Hom _{\tC} (\tC(i, i), V) \neq 0 \}
\end{equation*}
or $-\infty$ if $\td(V)=\emptyset$. In the latter case we say that $V$ is \emph{torsionless}. Its $0$-th \emph{homology} is defined to be
\begin{equation*}
H_0 (V) = V/JV \cong \tC_0 \otimes _{\tC} V.
\end{equation*}
Since $\tC_0 \otimes _{\tC} -$ is right exact, we define the $s$-th \emph{homology}
\begin{equation*}
H_s (V) = \Tor _s^{\tC} ( \tC_0, V)
\end{equation*}
for $s \geqslant 1$. Note that this is a $\tC$-module since $\tC_0$ is a $(\tC, \tC)$-bimodule. Moreover, it is finitely generated and torsion. For $s \in \Z$, the $s$-th \emph{homological degree} is set to be
\begin{equation*}
\hd_s(V) = \td(H_s(V)).
\end{equation*}
Sometimes we call the $0$-th homological degree \emph{generating degree}, and denote it by $\gd(V)$.

\begin{remark} \normalfont
The above definition of torsion degrees seems mysterious, so let us give an equivalent but more concrete definition. Let $V$ be a finitely generated $\C$-module and $i \in \Z$ be an object. If there exists a nonzero $v \in V_i$ and $\alpha_0 \in \C(i, i+1)$ such that $\alpha_0 \cdot v = 0$, we claim that for all $\alpha \in \C(i, i+1)$, one has $\alpha \cdot v = 0$. Indeed, since the symmetric group $S_{i+1} = \C(i+1, i+1)$ acts transitively on $\C(i, i+1)$ from the left side, for an arbitrary $\alpha \in \C(i, i+1)$, we can find an element $g \in \C(i+1, i+1)$ (which is unique) such that $\alpha = g \alpha_0$. Therefore, $\alpha \cdot v = g \alpha_0 \cdot v = 0$.

This observation tells us that $\alpha_0$ (and hence all $\alpha \in \C(i, i+1)$) sends the $\C(i,i)$-module $\tC(i, i) \cdot v$ to 0. Indeed, for every $g \in \C(i, i)$, since $\alpha_0 g \in \C(i, i+1)$, one has $(\alpha_0  g) \cdot v = 0$ by the argument in the previous paragraph. But the $\C(i,i)$-module $\tC(i, i) \cdot v$ can be regarded as a $\C$-module in a natural way, so $\Hom _{\tC} (\tC(i, i), V) \neq 0$. Conversely, if $\Hom _{\tC} (\tC(i, i), V) \neq 0$, one can easily find a nonzero element $v \in V_i$ such that $\alpha \cdot v = 0$ for $\alpha \in \C(i, i+1)$.

The above observations  immediately imply
\begin{equation*}
\td(V) = \sup \{ i \in \Z \mid \exists \, 0 \neq v \in V_i \text{ and } \alpha \in \C(i, i+1) \text{ such that } \alpha \cdot v = 0 \}.
\end{equation*}
\end{remark}

\begin{remark} \normalfont
Homologies of $\FI$-modules were defined to be homologies of a special complex in \cite[Section 2.4]{CEFN}. Gan and the first author proved in \cite{GL3} that homologies of this special complex coincide with ones defined in the above way. Since $\Tor$ is a classical homological construction, in this paper we take the above definition.
\end{remark}

Since each $\mk S_i = \tC(i, i)$ is a subalgebra of $\tC$ for $i \in \Z$, given a $\mk S_i$-module $T$, it induces a $\tC$-module $\tC \otimes _{\mk S_i} T$. We call these modules \emph{basic $\sharp$-filtered modules}. A finitely generated $\C$-module $V$ is called \emph{$\sharp$-filtered} by Nagpal if it has a filtration
\begin{equation*}
0 = V^{-1} \subseteq V^0 \subseteq \ldots \subseteq V^n = V
\end{equation*}
such that $V^{i+1}/V^i$ is isomorphic to a basic $\sharp$-filtered module for $-1 \leqslant i \leqslant n-1$; see \cite[Definition 1.10]{N}. The reader will see that $\sharp$-filtered modules have similar homological behaviors as projective modules.

\subsection{Main results}

Now we are ready to state main results of this paper. The first result characterizes $\sharp$-filtered modules by homological degrees.

\begin{theorem}[Homological characterizations of $\sharp$-filtered modules] \label{first main result}
Let $\mk$ be a commutative Noetherian ring and let $V$ be a finitely generated $\C$-module. Then the following statements are equivalent:
\begin{enumerate}
\item $V$ is $\sharp$-filtered;
\item $\hd_s(V) = - \infty$ for all $s \geqslant 1$;
\item $\hd_1(V) = - \infty$;
\item $\hd_s(V) = -\infty$ for some $s \geqslant 1$.
\end{enumerate}
\end{theorem}

\begin{remark} \normalfont This theorem was also independently proved almost at the same time by Ramos in \cite[Theorem B]{R} via a different approach.
\end{remark}

Using these homological characterizations, one can deduce an upper bound for projective dimensions of finitely generated $\C$-modules whose projective dimension is finite.

\begin{theorem}[Upper bounds of projective dimensions] \label{second main result}
Let $\mk$ be a commutative Noetherian ring whose finitistic dimension $\fdim \mk$ is finite \footnote{By definition, the finitistic dimension is the supremum of projective dimensions of finitely generated $\mk$-modules whose projective dimension is finite. The famous finitistic dimension conjecture asserts that if $\mk$ is a finite dimensional algebra, then $\fdim \mk < \infty$. However, the finitistic dimension of an arbitrary commutative Noetherian ring might be infinity.}, and let $V$ be a finitely generated $\C$-module with $\gd(V) = n$. Then the projective dimension $\pd (V)$ is finite if and only if for $0 \leqslant i \leqslant n$, one has
\begin{equation*}
V^i / V^{i-1} \cong \tC \otimes _{\mk S_i} (V^i / V^{i-1})_i
\end{equation*}
and
\begin{equation*}
\pd _{\mk S_i} ((V^i / V^{i-1})_i) < \infty,
\end{equation*}
where $V^i$ is the submodule of $V$ generated by $\bigoplus _{j \leqslant i} V_j$, Moreover, in that case
\begin{equation*}
\pd(V) = \max \{ \pd _{\mk S_i} ((V^i / V^{i-1})_i) \} _{i=0}^n = \max \{ \pd _{\mk} ((V^i / V^{i-1})_i) \} _{i=0}^n \leqslant \fdim \mk.
\end{equation*}
\end{theorem}

\begin{remark} \normalfont
This theorem asserts that finitely generated $\C$-modules which are not $\sharp$-filtered have infinite projective dimension. Moreover, if the finitistic dimension of $\mk$ is 0, or in particular the global dimension $\gldim \mk$ is 0, then the projective dimension of a $\C$-module is either 0 or infinity. This special result has been pointed out in \cite[Corollary 1.6]{GL2} for fields of characteristic 0.
\end{remark}

Another important application of Theorem \ref{first main result} is to prove the fact that every finitely generated $\C$-module can be approximated by a finite complex of $\sharp$-filtered modules, which was firstly proved by Nagpal in \cite[Theorem A]{N}. We give a new proof based on the conclusion of Theorem \ref{first main result} as well as the shift functor.

\begin{theorem}[$\sharp$-Filtered complexes] \label{sharp-filtered complex}
Let $\mk$ be a commutative Noetherian ring and let $V$ be a finitely generated $\C$-module. Then there exists a complex
\begin{equation*}
F^{\bullet}: \quad 0 \to V \to F^0 \to F^1 \to \ldots \to F^n \to 0
\end{equation*}
satisfying the following conditions:
\begin{enumerate}
\item each $F^i$ is a $\sharp$-filtered module with $\gd(F^i) \leqslant \gd(V) - i$;
\item $n \leqslant \gd(V)$;
\item the homology in each degree of the complex is a torsion module, including the homology at $V$.
\end{enumerate}
In particular, $\Sigma_d V$ is a $\sharp$-filtered module for $d \gg 0$.
\end{theorem}

\begin{remark} \normalfont
This complex of $\sharp$-filtered modules generalizes the finite injective resolution described in \cite{GL2}, where it was used by Gan and the first author to give a homological proof for the uniform representation stability phenomenon observed and proved in \cite{CEF, CF}. For an arbitrary field, it also implies the polynomial growth of finitely generated $\C$-modules; see \cite[Theorem B]{CEFN} and Remark \ref{polynomial growth}.
\end{remark}

Homological characterizations of $\sharp$-filtered modules also play a very important role in estimating homological degrees of finitely generated $\C$-modules $V$. Relying on an existing upper bound described in \cite[Theorem A]{CE}, we obtain another upper bound for homological degrees of $\C$-modules, removing the unnecessary assumption that $\mk$ is a field of characteristic 0 in \cite[Theorem 1.17]{Li2}.

\begin{theorem}[Castlenuovo-Mumford regularity] \label{third main result}
Let $\mk$ be a commutative Noetherian ring and let $V$ be a finitely generated $\C$-module. Then for $s \geqslant 1$,
\begin{equation}
\hd_s(V) \leqslant \max \{ 2\gd(V) - 1, \, \td(V) \} + s.
\end{equation}
\end{theorem}

\begin{remark} \normalfont
Theorem A in \cite{CE} asserts that
\begin{equation*}
\hd_s(V) \leqslant \hd_0(V) + \hd_1(V) + s - 1
\end{equation*}
for $s \geqslant 1$. The conclusion of the above theorem refines this bound for torsionless modules since in that case $\td(V) = - \infty$ and Corollary \ref{reduction} tells us that by a certain reduction one can always assume that $\gd(V) < \hd_1(V)$. Furthermore, it is more practical since it is easier to find $\td(V)$ than $\hd_1(V)$. The reader may refer to \cite[Example 5.20]{Li2}.
\end{remark}

\begin{remark} \normalfont
Given a finitely generated $\C$-module $V$, there exists a short exact sequence
\begin{equation*}
0 \to V_T \to V \to V_F \to 0
\end{equation*}
such that $V_T$ is torsion and $V_F$ is torsionless. Note that $\gd(V_F) \leqslant \gd(V)$ and $\td(V_T) = \td(V)$. Using the long exact sequence induced by this short exact sequence, one intuitively sees that the torsion part $V_T$ contributes to the term $\td(V)$ in inequality (1.1), and $V_F$ contributes to the term $2\gd(V) - 1$ in inquality (1.1). Indeed, if $V$ is a torsionless $\C$-module, then we have $\hd_s(V) \leqslant 2\gd(V) + s - 1$ for $s \geqslant 1$. On the other hand, if $V$ is a torsion module, then $\hd_s(V) \leqslant \td(V) + s$ for $s \geqslant 0$.
\end{remark}

\subsection{Organization} The paper is organized as follows. In Section 2 we introduce some elementary but important properties of the shift functor $\Sigma$ and its induced cokernel functor $D$. In particular, if $V$ is a torsionless $\C$-module, then an \emph{adaptable projective resolution} gives rise to an adaptable projective resolution of $DV$; see Definition \ref{adaptable projective resolutions} and Proposition \ref{degrees of projective resolutions}. This observation provides us a useful technique to estimate homological degrees of $V$. Those $\sharp$-filtered modules are studied in details in Section 3. We characterize $\sharp$-filtered modules by homological degrees, and use it to prove that every finitely generated $\C$-module becomes $\sharp$-filtered after applying the shift functor enough times. In the last section we prove all main results mentioned before.

\section{Preliminary results}

Throughout this section let $\mk$ be a commutative Noetherian ring, and let $\C$ be the skeletal subcategory of $\FI$ with objects $[n]$, $n \in \Z$.

\subsection{Functor $\Sigma$}

The shift functor $\Sigma: \C \module \to \C \module$ has been defined in the previous section. We list certain properties.

\begin{proposition}
Let $V$ be a $\C$-module. Then one has:
\begin{enumerate}
\item $\Sigma (\tC (i, -)) \cong \tC(i,-) \oplus \tC(i-1, -)^{\oplus i}$.
\item If $\gd(V) \leqslant n$, then $\gd (\Sigma V) \leqslant n$; conversely, if $\gd(\Sigma V) \leqslant n$, then $\gd(V) \leqslant n+1$.
\item The $\C$-module $V$ is finitely generated if and only if so is $\Sigma V$.
\item If $V$ is torsionless, so is $\Sigma V$.
\end{enumerate}
\end{proposition}

\begin{proof}
Statement (1) is well known, and it immediately implies the first half of (2) since $\Sigma$ is an exact functor. For the second half of (2), one may refer to the proof of \cite[Lemma 3.4]{Li2}. Statement (3) is immediately implied by (2) as a $\C$-module is finitely generated if and only if its generating degree is finite and its value on each object is a finitely generated $\mk$-module.

To prove (4), one observes that $V$ is torsionless if and only if the following conditions hold: for $i \in \Z$, $0 \neq v \in V_i$, and $\alpha \in \C(i, i+1)$, one always has $\alpha \cdot v \neq 0$. If $\Sigma V$ is not torsionless, we can find a nonzero element $v \in (\Sigma V)_i$ and $\alpha \in \C(i, i+1)$ for a certain $i \in \Z$ such that $\alpha \cdot v = 0$. But $(\Sigma V)_i = V_{i+1}$, and $\iota (\C(i, i+1)) \subseteq \C(i+1, i+2)$ where $\iota$ is the self-embedding functor inducing $\Sigma$. Therefore, by regarding $v$ as an element in $V_{i+1}$ one has $\iota (\alpha) \cdot v = 0$. Consequently, $V$ is not torsionless either. The conclusion follows from this contradiction.
\end{proof}

\subsection{Homological degrees under shift.}

The following lemma is a direct application of \cite[Proposition 4.5]{Li2}.

\begin{lemma} \label{hd under shift}
Let $V$ be a finitely generated $\C$-module. Then for $s \geqslant 0$,
\begin{equation*}
\hd_s (V) \leqslant \max \{\hd_0(V) + 1, \, \ldots, \, \hd_{s-1} (V) + 1, \, \hd_s (\Sigma V) + 1 \}.
\end{equation*}
\end{lemma}

If $V$ is a torsion module, then $\Sigma_d V = 0$ for a large enough $d$. Thus the above lemma can be used to estimate homological degrees of torsion modules.

\begin{proposition} \cite[Theorem 1.5]{Li2}  \label{hd of torsion modules}
If $V$ is a finitely generated torsion $\C$-module, then for $s \in \Z$, one has
\begin{equation*}
\hd_s(V) \leqslant \td(V) + s.
\end{equation*}
\end{proposition}

\subsection{Functor $D$}

The functor $D$ was introduced in \cite{CE, CEF}. Here we briefly mention its definition. Since the family of inclusions
\begin{equation} \label{pi}
\{ \pi_i: [i] \to [i+1], \, r \mapsto r+1 \mid i \geqslant 0 \}
\end{equation}
gives a natural transformation $\pi$ from the identity functor $\mathrm{Id}_{\C}$ to the self-embedding functor $\iota$, we obtain a natural transformation $\pi^{\ast}$ from the identity functor on $\C \module$ to $\Sigma$, which induces a natural map $\pi^{\ast}_V: V \to \Sigma V$ for each $\C$-module $V$. The functor $D$ is defined to be the cokernel of this map. Clearly, $D$ is a right exact functor. That is, it preserves surjection.

The following properties of $D$ play a key role in our approach.

\begin{proposition} \label{properties of D}
Let $D$ be the functor defined as above.
\begin{enumerate}
\item The functor $D$ preserves projective $\C$-modules. Moreover, $D \tC(i,-) \cong \tC(i-1, -) ^{\oplus i}$.
\item A $\C$-module $V$ is torsionless if and only if there is a short exact sequence
\begin{equation*}
0 \to V \to \Sigma V \to DV \to 0.
\end{equation*}
\item Let $V$ be a finitely generated $\C$-module. Then:
\begin{equation*}
\gd(DV) =
\begin{cases}
-\infty & \text{ if } \gd(V) = 0 \text{ or } -\infty \\
\gd(V) - 1 & \text{ if } \gd(V) \geqslant 1.
\end{cases}
\end{equation*}
\end{enumerate}
\end{proposition}

\begin{proof}
Statements (1) and (2) have been established in \cite[Lemma 3.6]{CE}, so we only give a proof of (3). Take a surjection $P \to V \to 0$ such that $P$ is a projective $\tC$-module and $\gd(P) = \gd(V) = n$. Since $D$ is a right exact functor, we get a surjection $DP \to DV \to 0$. When $n = 0$ or $-\infty$, we know that $DP = 0$, and hence $DV = 0$, so $\gd(DV) = - \infty$. If $n > 0$, then $DP$ is a projective module with $\gd(DP) = n - 1$ by (1). Consequently, $\gd(DV) \leqslant n - 1$. We finish the proof by showing that $\gd(DV) \geqslant n-1$ as well.

Let $V'$ be the submodule of $V$ generated by $\bigoplus _{i \leqslant n-1} V_i$. This is a proper submodule of $V$ since $\gd(V) = n$. Let $V'' = V/V'$, which is not zero. Moreover, $V''$ is a $\C$-module generated in degree $n$. Applying the right exact functor $D$ to $V \to V'' \to 0$ one gets a surjection $DV \to DV'' \to 0$. But one easily sees that $(DV'')_i = 0$ for $i < n-1$ and $(DV'')_{n-1} \neq 0$. Consequently, $\gd(DV'') \geqslant n-1$. This forces $\gd(DV) \geqslant n-1$.
\end{proof}

\begin{remark} \label{kernel of V to Sigma V} \normalfont
If $V$ is not torsionless, we have an exact sequence
\begin{equation*}
0 \to V_T \to V \to V_F \to 0
\end{equation*}
such that $V_T \neq 0$ is torsion and $V_F$ is torsionless. It induces a short exact sequence
\begin{equation*}
0 \to \Sigma V_T \to \Sigma V \to \Sigma V_F \to 0.
\end{equation*}
Using snake lemma, one obtains exact sequences
\begin{equation*}
0 \to K \to V \to \Sigma V \to DV \to 0
\end{equation*}
and
\begin{equation*}
0 \to K \to V_T \to \Sigma V_T \to DV_T \to 0.
\end{equation*}
In particular, $K$ is a torsion module.
\end{remark}

An immediate consequence is:

\begin{corollary} \label{torsionless modules induce diagram}
A short exact sequence $0 \to W \to M \to V \to 0$ of finitely generated torsionless $\C$-modules gives rise to the following commutative diagram such that all rows and columns are short exact sequences
\begin{equation*}
\xymatrix{
0 \ar[r] & W \ar[r] \ar[d] & \Sigma W \ar[r] \ar[d] & DW \ar[r] \ar[d] & 0\\
0 \ar[r] & M \ar[r] \ar[d] & \Sigma M \ar[r] \ar[d] & DM \ar[r] \ar[d] & 0\\
0 \ar[r] & V \ar[r] & \Sigma V \ar[r] & DV \ar[r] & 0.
}
\end{equation*}
\end{corollary}

\begin{proof}
Since $W$ and $M$ are torsionless, By \cite[Lemma 3.6]{CE}, $W \to \Sigma W$ and $M \to \Sigma M$ are injective, and one gets a commutative diagram
\begin{equation*}
\xymatrix{
0 \ar[r] & W \ar[r] \ar[d] & \Sigma W \ar[r] \ar[d] & DW \ar[r] \ar[d] ^{\alpha} & 0\\
0 \ar[r] & M \ar[r] & \Sigma M \ar[r] & DM \ar[r] & 0,
}
\end{equation*}
which by the snake lemma induces the following exact sequence
\begin{equation*}
0 \to \ker \alpha \to V \to \Sigma V \to DV \to 0.
\end{equation*}
But since $V$ is torsionless, the map $V \to \Sigma V$ is injective. Therefore, $\ker \alpha = 0$. The conclusion follows.
\end{proof}

The next lemma asserts that the functor $D$ ``almost" commutes with $\Sigma$.

\begin{lemma} \label{D and Sigma commute}
Let $V$ be a finitely generated $\C$-module. Then $\Sigma DV \cong D \Sigma V$.
\end{lemma}

\begin{proof}
It is sufficient to construct a natural isomorphism between $\Sigma D$ and $D \Sigma$. This has been done by Ramos in \cite[Lemma 3.5]{R}. Note that in the setting of that paper, the self-embedding functor is defined in a way different from ours; see \cite[Definition 2.20]{R}. Therefore, we have to slightly modify the  proof in \cite[Lemma 3.5]{R}. In our setting,
\begin{align*}
(\Sigma DV)_n & = V_{n+2} / \pi_{n+1} (V_{n+1});\\
(D \Sigma V)_n & = V_{n+2} / (\iota (\pi_n)) (V_{n+1})
\end{align*}
where $\iota$ is the self-embedding functor and $\pi_n$ is defined in (\ref{pi}). By (\ref{iota}) and (\ref{pi}), we have
\begin{align*}
\pi_{n+1}: [n+1] \to [n+2], \quad & i \mapsto i+1;\\
\iota(\pi_n): [n+1] \to [n+2], \quad & i \mapsto
\begin{cases}
1, & i = 1;\\
i+1, & 2 \leqslant i \leqslant n+1.
\end{cases}
\end{align*}
Now the reader can see that $(D\Sigma V)_n$ and $(\Sigma D V)_n$ are isomorphic under the action of an bijection $\alpha_{n+2}: [n+2] \to [n+2]$ which permutes $1$ and $2$ and fixes all other elements. Moreover, the family of such bijections $\{ \alpha_n \mid n \geqslant 0\}$ gives a natural isomorphism between $\Sigma D$ and $D \Sigma$.
\end{proof}

\subsection{Adaptable projective resolutions}

A standard way to compute homologies and hence homological degrees is to use a suitable projective resolution.

\begin{definition} \label{adaptable projective resolutions}
Let $V$ be a finitely generated $\tC$-module. A projective resolution
\begin{equation*}
\ldots \to P^s \to P^{s-1} \to \ldots \to P^0 \to V \to 0
\end{equation*}
of $V$ is said to be \emph{adaptable} if for every $s \geqslant -1$, $\gd(P^{s+1}) = \gd(Z^s)$, where $Z^s$ is the $s$-th cycle and by convention $Z^{-1} = V$.
\end{definition}

\begin{lemma} \label{compare degrees}
Let $0 \to W \to P \ \to V \to 0$ be a short exact sequence of finitely generated $\C$-modules such that $P$ is projective and $\gd(V) = \gd(P)$. Then
\begin{equation*}
\gd(W) \leqslant \max \{ \hd_1(V), \, \gd(V) \} = \max \{\gd(V), \, \gd(W) \}.
\end{equation*}
\end{lemma}

\begin{proof}
The conclusions hold for $V = 0$ by convention, so we assume that $V$ is nonzero. The given short exact sequence gives rise to
\begin{equation*}
0 \to H_1 (V) \to H_0(W) \to H_0(P) \to H_0(V) \to 0.
\end{equation*}
Clearly,
\begin{equation*}
\gd(W) = \td(H_0(W)) \leqslant \max\{\td(H_1(V)), \, \td(H_0(P)) \} = \max \{ \hd_1(V), \, \gd(V) \}.
\end{equation*}
Moreover, if $\hd_1(V) \leqslant \gd(V)$, then $\gd(W) \leqslant \gd(V)$ as well. If $\hd_1(V) > \gd(V)$, then $\hd_1(V) = \gd(W)$. The equality follows from this observation.
\end{proof}

Given a finitely generated $\C$-module $V$, the following corollary relates generating degrees of components in an adaptable projective resolution of $V$ to homological degrees of $V$. That is:

\begin{corollary} \label{estimate degrees recursively}
Let $V$ be a finitely generated $\tC$-module, and let
\begin{equation*}
\ldots \to P^s \to P^{s-1} \to \ldots \to P^0 \to V \to 0
\end{equation*}
be an adaptable projective resolution of $V$. Let $d_s = \gd(P^s)$. Then
\begin{equation*}
d_s \leqslant \max \{ \hd_0(V), \, \ldots, \, \hd_{s-1}(V), \, \hd_s(V) \}
\end{equation*}
and
\begin{equation*}
\max \{d_0, \, \ldots, \, d_{s-1}, \, d_s \} = \max \{ \hd_0(V), \, \ldots, \, \hd_{s-1}(V), \, \hd_s(V) \}.
\end{equation*}
\end{corollary}

\begin{proof}
We use induction on $s$. If $s = 0$, nothing needs to show. Suppose that the conclusion holds for $s = n \geqslant 0$, and consider $s = n+1$.

Consider the short exact sequence
\begin{equation*}
0 \to Z^n \to P^n \to Z^{n-1} \to 0.
\end{equation*}
By Lemma \ref{compare degrees} and the definition of adaptable projective resolutions,
\begin{equation*}
d_{n+1} = \gd(Z^n) \leqslant \max \{\gd(Z^{n-1}), \, \hd_1 (Z^{n-1}) \} = \max \{d_n, \, \hd_{n+1} (V) \}
\end{equation*}
since $V = Z^{-1}$ by convention. However, by induction,
\begin{equation*}
d_n \leqslant \max \{ \hd_0(V), \, \ldots, \, \hd_n(V)\}.
\end{equation*}
The last two inequalities imply the conclusion for $n+1$. This establishes the inequality.

To show the equality, one observes that the inequality we just proved implies that
\begin{equation*}
\max \{d_0, \, \ldots, \, d_{s-1}, \, d_s \} \leqslant \max \{ \hd_0(V), \, \ldots, \, \hd_{s-1}(V), \, \hd_s(V) \}.
\end{equation*}
However, one observes from the definition of homologies that $d_i \geqslant \hd_i (V)$ for $i \in \Z$. Thus we also have
\begin{equation*}
\max \{d_0, \, \ldots, \, d_{s-1}, \, d_s \} \geqslant \max \{ \hd_0(V), \, \ldots, \, \hd_{s-1}(V), \, \hd_s(V) \}.
\end{equation*}
\end{proof}

Using the functor $D$, one may relate the homological degrees of a finitely generated $\C$-module $V$ to those of $DV$.

\begin{proposition} \label{degrees of projective resolutions}
Let $V$ be a finitely generated torsionless $\tC$-module, and let
\begin{equation*}
\ldots \to P^s \to P^{s-1} \to \ldots \to P^0 \to V \to 0
\end{equation*}
be an adaptable projective resolution of $V$. Then it induces an adaptable projective resolution
\begin{equation*}
\ldots \to DP^s \to DP^{s-1} \to \ldots \to DP^0 \to DV \to 0
\end{equation*}
such that for $s \in \Z$
\begin{equation*}
\gd(DP^s) =
\begin{cases}
-\infty & \text{ if } \gd(P^s) = 0 \text{ or } -\infty\\
\gd(P^s) - 1 & \text{ if } \gd(P^s) \geqslant 1.
\end{cases}
\end{equation*}
\end{proposition}

\begin{proof}
Let $P^{\bullet} \to V \to 0$ be the resolution. Since $V$ and all cycles are torsionless, by Corollary \ref{torsionless modules induce diagram} one gets a commutative diagram
\begin{equation*}
\xymatrix{
0 \ar[r] & P^{\bullet} \ar[r] \ar[d] & \Sigma P^{\bullet} \ar[r] \ar[d] & DP^{\bullet} \ar[r] \ar[d] & 0\\
0 \ar[r] & V \ar[r] & \Sigma V \ar[r] & DV \ar[r] & 0.
}
\end{equation*}
The conclusion then follows from Proposition \ref{properties of D}.
\end{proof}

As an immediate consequence of the above result, we have:

\begin{corollary} \label{compare the first two degrees}
Let $V$ be a finitely generated torsionless $\C$-module. Then for $s \in \Z$
\begin{equation*}
\max \{\hd_0(V), \, \ldots, \, \hd_s(V) \} \geqslant \max \{\hd_0(DV), \, \ldots, \, \hd_s(DV) \} + 1.
\end{equation*}
Moreover, the equality holds if $\gd(V) \geqslant 1$.
\end{corollary}

\begin{proof}
The conclusion holds trivially if $\gd(V) = 0$ or $-\infty$ since in that case $DV = 0$. So we assume $\gd(V) \geqslant 1$, and $DV \neq 0$. Let $ P^{\bullet} \to V \to 0$ be an adaptable projective resolution. By Corollary \ref{estimate degrees recursively} and Proposition \ref{degrees of projective resolutions}, one has
\begin{equation*}
\max \{\gd(V), \, \ldots, \, \hd_s(V) \} = \max \{\gd(P^0), \, \ldots, \, \gd(P^s) \}
\end{equation*}
and
\begin{equation*}
\max \{\gd(DV), \, \ldots, \, \hd_s(DV) \} = \max \{\gd(DP^0), \, \ldots, \, \gd(DP^s) \}.
\end{equation*}
Moreover, $\gd(P^i) \geqslant \gd(DP^i) + 1$ for $i \in \Z$, and the equality holds if $\gd(P^i) \geqslant 1$. The desired inequality and equality follow from these observations.
\end{proof}

\section{Filtrations of $\FI$-modules}

In the previous section we use adaptable projective resolutions to estimate homological degrees of finitely generated $\C$-modules. However, since finitely generated $\C$-modules usually have infinite projective dimension, the resolutions are of infinite length. For the purpose of estimating homological degrees, $\sharp$-filtered modules play a more subtle role since we will show that every finitely generated $\C$-module $V$ gives rise to a complex of $\sharp$-filtered modules which is of finite length. Moreover, we will see that these special modules, coinciding with projective modules when $\mk$ is a field of characteristic 0, have similar homological properties as projective modules.

\subsection{A homological characterization of $\sharp$-filtered modules}

Recall that a finitely generated $\C$-module is \emph{$\sharp$-filtered} if there exists a chain
\begin{equation*}
0 = V^{-1} \subseteq V^0 \subseteq \ldots \subseteq V^n = V
\end{equation*}
such that $V^i /V^{i-1}$ is isomorphic to $\tC \otimes _{\mk S_i} T_i$ for $0 \leqslant i \leqslant n$, where $S_i$ is the symmetric group on $i$ letters, and $T_i$ is a finitely generated $\mk S_i$-module.

An important fact of $\FI$, which can be easily observed, is:

\begin{lemma}
For $n \in \Z$, the $\C$-module $\tC 1_n$ is a right free $\mk S_n$-module.
\end{lemma}

\begin{proof}
Note that for $n \in \Z$ and $m \geqslant n$, the group $S_n = \C(n, n)$ acts freely on $\C(n, m)$ from the right side. The conclusion follows.
\end{proof}

This elementary observation implies that higher homologies of $\sharp$-filtered modules vanish.

\begin{lemma} \label{higher homologies of sharp-filtered modules vanish}
If $V$ is a $\sharp$-filtered module, then $H_s(V) = 0$ for all $s \geqslant 1$.
\end{lemma}

\begin{proof}
Firstly we consider a special case: $V$ is basic. That is, $V = \tC \otimes _{\mk S_i} V_i$ for some $i \in \Z$. Let
\begin{equation*}
0 \to W_i \to P_i \to V_i \to 0
\end{equation*}
be a short exact sequence of $\mk S_i$-modules such that $P_i$ is projective. Since $\tC$ is a right projective $\mk S_i$-module, we get an exact sequence
\begin{equation*}
0 \to W = \tC \otimes _{\mk S_i} W_i \to P = \tC \otimes _{\mk S_i} P_i \to V = \tC \otimes _{\mk S_i} V_i \to 0.
\end{equation*}
Note that the middle term is a projective $\tC$-module. By applying $\tC_0 \otimes_{\tC} -$ one recovers the first exact sequence, so $H_1(V) = 0$. Replacing $V$ by $W$, one deduces that $H_2(V) = H_1(W) = 0$. The conclusion follows by recursion.

For the general case, one may take a filtration for $V$, each component of which is a basic $\sharp$-filtered module. The conclusion follows from a standard homological method: short exact sequences induce long exact sequences on homologies.
\end{proof}

The following lemma was proved in \cite[Lemma 2.2]{N}. Here we give two proofs from the homological viewpoint.

\begin{lemma} \label{modules generated in one degree}
Let $V$ be a finitely generated $\C$-module generated in one degree. If $\hd_1(V) \leqslant \gd(V)$, then $V$ is a $\sharp$-filtered module.
\end{lemma}

\begin{proof}
The conclusion holds trivially for $V=0$, so we assume $\gd(V) = n \geqslant 0$. Consider the short exact sequence
\begin{equation*}
0 \to W \to P \to V \to 0
\end{equation*}
where $P$ is a projective $\C$-module with $\gd(P) = n$. Since $\hd_1(V) \leqslant n$, one knows that $\gd(W) \leqslant n$ by Lemma \ref{compare degrees}. If $\gd(W) < n$, then $W = 0$ since $W_i = 0$ for all $i < n$. Thus $V \cong P$ is clearly a $\sharp$-filtered module. Now we consider the case that $\gd(W) = n$.

Since $0 \to W_n \to P_n \to V_n \to 0$ is a short exact sequence of $\mk S_n$-modules and $\tC$ is a right projective $\mk S_n$-module, we obtain a short exact sequence
\begin{equation*}
0 \to \tC \otimes _{\mk S_n} W_n \to \tC \otimes _{\mk S_n} P_n \to \tC \otimes _{\mk S_n} V_n \to 0.
\end{equation*}
Note that $W$, $P$, and $V$ are all generated in degree $n$. Via the multiplication map we get a commutative diagram such that all vertical maps are surjective:
\begin{equation*}
\xymatrix{
0 \ar[r] & \tC \otimes _{\mk S_n} W_n \ar[r] \ar[d] & \tC \otimes _{\mk S_n} P_n \ar[r] \ar[d] & \tC \otimes _{\mk S_n} V_n \ar[r] \ar[d] & 0\\
0 \ar[r] & W \ar[r] & P \ar[r] & V \ar[r] & 0.
}
\end{equation*}
But the middle vertical map is actually an isomorphism. This forces the other two vertical maps to be isomorphisms by snake lemma, and the conclusion follows.
\end{proof}

\begin{proof}
If $V = 0$, nothing needs to show. Otherwise, let $\gd(V) = n \geqslant 0$. Since $V$ is generated in degree $n$, there is a short exact sequence
\begin{equation*}
0 \to K \to \tilde{V} = \tC \otimes _{\mk S_n} V_n \to V \to 0
\end{equation*}
which induces an exact sequence
\begin{equation*}
0 \to H_1(V) \to H_0(K) \to H_0(\tilde{V}) \to H_0(V) \to 0
\end{equation*}
by the previous lemma. Note that the map $H_0(\widetilde{V}) \to H_0(V)$ is an isomorphism. Consequently, $H_0(K) \cong H_1(V)$. In particular, $\gd(K) = \hd_1(V) \leqslant \gd(V) = n$. But it is clear that $K_i = 0$ for all $i \leqslant n$. Therefore, the only possibility is that $K = 0$.
\end{proof}

A useful result is:

\begin{corollary} \label{reduction}
Let $V$ be a finitely generated $\C$-module. Then there exists a short exact sequence
\begin{equation*}
0 \to U \to V \to W \to 0
\end{equation*}
such that $H_s(U) = H_s(V)$ for $s \geqslant 1$ and $W$ is $\sharp$-filtered. Moreover, if $V$ is not $\sharp$-filtered, one always has $\hd_1(U) > \gd(U)$.
\end{corollary}

\begin{proof}
Suppose that $V$ is nonzero. If $\hd_1(V) > \gd(V)$, then we can let $U = V$ and $W = 0$, so the conclusion holds trivially. Otherwise, one has a short exact sequence
\begin{equation*}
0 \to V' \to V \to V'' \to 0
\end{equation*}
where $V'$ is the submodule of $V$ generated by $\bigoplus _{i \leqslant \gd(V) - 1} V_i$, which might be 0. Then $V'' \neq 0$. The long exact sequence
\begin{equation*}
\ldots \to H_1(V) \to H_1(V'') \to H_0(V') \to H_0(V) \to H_0(V'') \to 0
\end{equation*}
tells us that
\begin{equation*}
\hd_1(V'') \leqslant \max \{ \gd(V'), \, \hd_1(V) \} \leqslant \gd(V) = \gd(V'').
\end{equation*}
But the previous lemma asserts that $V''$ is $\sharp$-filtered.

If $\gd(V') < \hd_1(V')$, then the above short exact sequence is what we want. Otherwise, we can continue this process for $V'$. Since $\gd(V') < \gd(V)$, it must stop after finitely many steps. Since $V$ is not $\sharp$-filtered, in the last step we must get a submodule $U$ of $V$ with $\hd_1(U) > \gd(U)$. Moreover, since $W$ is $\sharp$-filtered, one easily deduces that $H_s(U) \cong H_s(V)$ for $s \geqslant 1$.
\end{proof}

A finitely generated $\C$-module is $\sharp$-filtered if and only if its higher homologies vanish, and if and only if its first homology vanishes.

\begin{theorem} \label{characterizations of sharp-filtered modules}
Let $\mk$ be a commutative Noetherian ring and let $V$ be a finitely generated $\C$-module. Then the following statements are equivalent:
\begin{enumerate}
\item $V$ is a $\sharp$-filtered module;
\item $H_i(V) = 0$ for all $i \geqslant 1$;
\item $H_1(V) = 0$.
\end{enumerate}
\end{theorem}

\begin{proof}
$(1) \Rightarrow (2)$: This is Lemma \ref{higher homologies of sharp-filtered modules vanish}.

$(2) \Rightarrow (3)$: Trivial.

$(3) \Rightarrow (1)$: Suppose that $H_1(V) = 0$. That is, $\hd_1(V) = - \infty$. If $V$ is not $\sharp$-filtered, then by Corollary \ref{reduction}, there exists a short exact sequence
\begin{equation*}
0 \to U \to V \to W \to 0
\end{equation*}
such that $H_1(U) \cong H_1(V) = 0$ and $\hd_1(U) > \gd(U)$. This is absurd.
\end{proof}

An extra bonus of this characterization is that: the category of finitely generated $\sharp$-filtered modules is closed under taking kernels and extensions.

\begin{corollary} \label{kernels of sharp-filtered modules are sharp-filtered}
Let $0 \to U \to V \to W \to 0$ be a short exact sequence.
\begin{enumerate}
\item If both $V$ and $W$ are $\sharp$-filtered, so is $U$.
\item If both $U$ and $W$ are $\sharp$-filtered, so is $V$.
\end{enumerate}
\end{corollary}

\begin{proof}
Use the long exact sequence induced by the given short exact sequence and the above theorem.
\end{proof}

\subsection{Properties of $\sharp$-filtered modules}

In this subsection we explore certain important properties of $\sharp$-filtered modules.

For a $\sharp$-filtered module $V$, one knows that it has a filtration by basic $\sharp$-filtered modules from the definition. The following result tells us an explicit construction of such a filtration.

\begin{proposition} \label{structure of sharp-filtered modules}
Let $V$ be a $\sharp$-filtered module with $\gd(V) = n$. Then there exists a chain of $\C$-modules
\begin{equation*}
0 = V^{-1} \subseteq V^0 \subseteq V^1 \subseteq \ldots \subseteq V^n = V
\end{equation*}
such that for $-1 \leqslant s \leqslant n-1$, $V^s$ is the submodule of $V$ generated by $\bigoplus _{i \leqslant s} V_i$ and $V^{s+1} / V^s$ is 0 or a basic $\sharp$-filtered module.
\end{proposition}

\begin{proof}
We prove the statement by induction on $\gd(V)$. The conclusion holds obviously for $n = 0$ or $n = -\infty$. Suppose that $n \geqslant 1$. We have a short exact sequence
\begin{equation*}
0 \to V' \to V \to V'' \to 0
\end{equation*}
such that $V'$ is the submodule generated by $\bigoplus _{i \leqslant n-1} V_i$. Then $V''$ is generated in degree $n$. The conclusion follows immediately after we show that $V''$ is a $\sharp$-filtered module. However, we can apply the same argument as in the proof of Corollary \ref{reduction}.
\end{proof}

\begin{remark} \normalfont
So far we do not use any special property of the category $\FI$ to prove the above results in this section. Therefore, all these results hold for $\mk$-linear categories $\underline {\mathcal{D}}$ satisfying the following conditions:
\begin{enumerate}
\item objects of $\underline {\mathcal{D}}$ are parameterized by nonnegative integers;
\item there is no nonzero morphisms from bigger objects to smaller objects;
\item $\underline {\mathcal{D}}$ is \emph{locally finite}; that is, $\underline {\mathcal{D}} (x, y)$ are finitely generated $\mk$-modules for $x, y \in \Z$;
\item $\underline {\mathcal{D}} (x, y)$ is a right projective $\underline {\mathcal{D}} (x, x)$-module for $x, y \in \Z$.
\end{enumerate}
\end{remark}

Now we begin to use some special properties of  $\FI$ to deduce more results. The following proposition tells us that the property of being $\sharp$-filtered is preserved by functors $\Sigma$ and $D$.

\begin{proposition} \label{sharp-filtered modules are torsionless}
Let $V$ be a finitely generated $\C$-module. If $V$ is $\sharp$-filtered, then it is torsionless. Moreover, $\Sigma V$ and $DV$ are also $\sharp$-filtered.
\end{proposition}

\begin{proof}
Clearly, we can suppose that $V$ is nonzero. Since $V$ is a $\sharp$-filtered module, it has a filtration such that each component of which has the form $\tC \otimes _{\mk S_i} T_i$, where each $T_i$ is a finitely generated $\mk S_i$-module. Since $V$ is a torsionless module if and only if $\Hom _{\tC} (\tC_0, V)$ is 0, to show that $V$ is torsionless, it suffices to verify $\Hom _{\tC} (\tC_0, \tC \otimes _{\mk S_i} T_i) = 0$; that is, each $\tC \otimes _{\mk S_i} T_i$ is torsionless. By \cite[Lemma 2.2]{N}, there is a natural embedding
\begin{equation*}
0 \to \tC \otimes _{\mk S_i} M_i \to \Sigma (\tC \otimes _{\mk S_i} M_i).
\end{equation*}
By Proposition \ref{properties of D}, this $\sharp$-filtered module must be torsionless.

To prove the second statement, we use induction on the generating degree $n = \gd(V)$. For $n = 0$, the conclusion is implied by \cite[Lemma 2.2]{N}. For $n \geqslant 1$, one considers the short exact sequence
\begin{equation*}
0 \to V' \to V \to V'' \to 0
\end{equation*}
where $V'$ is the submodule of $V$ generated by $\bigoplus _{i \leqslant n-1} V_i$. By Proposition \ref{structure of sharp-filtered modules}, all terms in this short exact sequence are $\sharp$-filtered. Moreover, $V''$ is a basic $\sharp$-filtered module. Since we just proved that they are all torsionless, by Corollary \ref{torsionless modules induce diagram}, we have a commutative diagram each row or column of which is an short exact sequence:
\begin{equation*}
\xymatrix{
0 \ar[r] & V' \ar[r] \ar[d] & \Sigma V' \ar[r] \ar[d] & DV' \ar[r] \ar[d] & 0\\
0 \ar[r] & V \ar[r] \ar[d] & \Sigma V \ar[r] \ar[d] & DV \ar[r] \ar[d] & 0\\
0 \ar[r] & V'' \ar[r] & \Sigma V'' \ar[r] & DV'' \ar[r] & 0.
}
\end{equation*}
By induction hypothesis, $\Sigma V'$ is $\sharp$-filtered. Moreover, $\Sigma V''$ is $\sharp$-filtered by \cite[Lemma 2.2]{N}, so $\Sigma V$ is $\sharp$-filtered as well. The same reasoning tells us that $DV$ are also $\sharp$-filtered.
\end{proof}

\subsection{Finitely generated $\C$-modules become $\sharp$-filtered under enough shifts.}

The main task of this subsection is to use functor $D$ as well as the homological characterization of $\sharp$-filtered modules to show that for every finitely generated $\C$-module $V$, $\Sigma_N V$ is $\sharp$-filtered for $N \gg 0$.

As the starting point, we consider $\C$-modules generated in degree 0.

\begin{lemma} \label{modules with generating degree 0}
Every torsionless $\C$-module generated in degree 0 is $\sharp$-filtered.
\end{lemma}

\begin{proof}
Since $V$ is torsionless, one gets a short exact sequence
\begin{equation*}
0 \to V \to \Sigma V \to DV \to 0.
\end{equation*}
By Proposition \ref{properties of D}, $DV = 0$, so $V \cong \Sigma V$. In particular, one has
\begin{equation*}
V_0 \cong V_1 \cong V_2 \cong \ldots.
\end{equation*}
Thus $V \cong \tC \otimes _{\tC(0,0)} V_0$.
\end{proof}

\begin{lemma} \label{technical lemma}
Let $V$ be a finitely generated torsionless $\C$-module. If $DV$ is $\sharp$-filtered, then $\gd(V) \geqslant \hd_1(V)$.
\end{lemma}

\begin{proof}
One may assume that $V$ is nonzero. If $\gd(V) = 0$, then $V$ is $\sharp$-filtered by Lemma \ref{modules with generating degree 0} and the conclusion holds. For $\gd(V) \geqslant 1$, since $\hd_1(DV) = - \infty$, by Corollary \ref{compare the first two degrees},
\begin{equation*}
\max \{ \gd(V), \, \hd_1(V) \} = \max \{ \gd(DV) + 1, \, \hd_1(DV) + 1 \} = \gd(DV) + 1 = \gd(V),
\end{equation*}
where the last equality follows from Proposition \ref{properties of D}. This implies the desired result.
\end{proof}

The following lemma, similar to Proposition \ref{degrees of projective resolutions}, shows the connections between  a finitely generated $\C$-module $V$ and $DV$.

\begin{lemma} \label{crucial recursion}
Let $V$ be a finitely generated torsionless $\C$-module. If $DV$ is $\sharp$-filtered, so is $V$.
\end{lemma}

\begin{proof}
We use induction on $\gd(V)$. The conclusion for $\gd(V) = 0$ has been established in Lemma \ref{modules with generating degree 0}. Suppose that it holds for all finitely generated $\C$-modules with generating degree at most $n$, and let $V$ be a finitely generated $\C$-module with $\gd(V) = n+1$. As before, consider the exact sequence $0 \to V' \to V \to V'' \to 0$ where $V'$ is the submodule generated by $\bigoplus _{i \leqslant n} V_i$.

By considering the long exact sequence
\begin{equation*}
\ldots \to H_1(V) \to H_1(V'') \to H_0(V') \to H_0(V) \to H_0(V'') \to 0
\end{equation*}
and using the fact $\gd(V') \leqslant n$ and $n+1= \gd(V) \geqslant \hd_1(V)$ which is proved in Lemma \ref{technical lemma}, one deduces that $\hd_1(V'') \leqslant n+1 = \gd(V'')$. Consequently, $V''$ is a $\sharp$-filtered module by Lemma \ref{modules generated in one degree}, and hence is torsionless by Proposition \ref{sharp-filtered modules are torsionless}. Moreover, $DV''$ is $\sharp$-filtered as well by Proposition \ref{sharp-filtered modules are torsionless}.

Since $V''$ is torsionless, the above exact sequence gives rise to a commutative diagram
\begin{equation*}
\xymatrix{
0 \ar[r] & V' \ar[r] \ar[d] & \Sigma V' \ar[r] \ar[d] & DV' \ar[r] \ar[d] & 0\\
0 \ar[r] & V \ar[r] \ar[d] & \Sigma V \ar[r] \ar[d] & DV \ar[r] \ar[d] & 0\\
0 \ar[r] & V'' \ar[r] & \Sigma V'' \ar[r] & DV'' \ar[r] & 0.
}
\end{equation*}
Note that both $DV''$ and $DV$ are $\sharp$-filtered. By Corollary \ref{kernels of sharp-filtered modules are sharp-filtered}, $DV'$ is $\sharp$-filtered as well. But clearly $V'$ is torsionless since it is a submodule of the torsionless module $V$. By induction hypothesis, $V'$ is $\sharp$-filtered, so is $V$. The conclusion follows from induction.
\end{proof}

Now we are ready to show the following result.

\begin{theorem} \label{modules become sharp-filtered}
Let $\mk$ be a commutative Noetherian ring and let $V$ be a finitely generated $\C$-module. If $d \gg 0$, then $\Sigma_d V$ is $\sharp$-filtered.
\end{theorem}

\begin{proof}
Again, assume that $V$ is nonzero. Since $d \gg 0$, one may suppose that $\td(V) < d$. Therefore, $\Sigma_d V$ is torsionless. We use induction on $\gd(V)$. If $\gd(V) = 0$, then $\Sigma_d V$ is generated in degree 0, so the conclusion holds by Lemma \ref{modules with generating degree 0}. Now suppose that the conclusion holds for all modules with generating degrees at most $n$. We then deal with $\gd(V) = n+1$.

Consider the exact sequence
\begin{equation*}
0 \to \Sigma_d V \to \Sigma_{d+1} V \to \Sigma_d DV \cong D\Sigma_d V \to 0.
\end{equation*}
If $\gd(\Sigma_d V) = 0$, nothing needs to show. So we suppose that $\gd (\Sigma_d V) \geqslant 1$. Note that $\gd(DV) = n$. Therefore, by induction hypothesis, $\Sigma_d DV$ is $\sharp$-filtered. That is, $D \Sigma_d V$ is $\sharp$-filtered. By Lemma \ref{crucial recursion}, $\Sigma_d V$ is $\sharp$-filtered as well.
\end{proof}

\begin{remark} \normalfont \label{polynomial growth}
This theorem gives another proof for the polynomial growth phenomenon observed in \cite{CEFN}. Since for a sufficiently large $N \in \Z$, $\Sigma_N V$ has a filtration each component of which is exactly of the form $M(W)$ as in \cite[Definition 2.2.2]{CEF}, and each $M(W)$ satisfies the polynomial growth property, so is $\Sigma_N V$. Consequently, $\tau_N V$ satisfies the polynomial growth property.
\end{remark}

Because $\sharp$-filtered modules coincide with projective modules when $\mk$ is a field of characteristic 0, one has:

\begin{corollary}
Let $\mk$ be a field of characteristic 0 and let $V$ be a finitely generated $\C$-module. Then for $d \gg 0$, $\Sigma_d V$ is projective.
\end{corollary}

\section{Proofs of main results}

In this section we prove several main results mentioned in Section 1.

\subsection{Proof of Theorem \ref{first main result}}

Theorem \ref{characterizations of sharp-filtered modules} has established the equivalence of the first three statements in Theorem \ref{first main result}. In this subsection we show the equivalence between the first statement and the last one.

\begin{lemma}
Let $V$ be a finitely generated $\C$-module. If $H_2(V) = 0$, then $V$ is torsionless.
\end{lemma}

\begin{proof}
Consider a short exact sequence
\begin{equation*}
0 \to W \to P \to V \to 0.
\end{equation*}
Since $H_1(W) = H_2(V) = 0$, $W$ is $\sharp$-filtered. Applying $\Sigma$ to it one gets
\begin{equation*}
0 \to \Sigma W \to \Sigma P \to \Sigma V \to 0.
\end{equation*}
They induce the following commutative diagram:
\begin{equation*}
\xymatrix{
0 \ar[r] & W \ar[r] \ar[d] & \Sigma W \ar[r] \ar[d] & DW \ar[r] \ar[d]^{\alpha} & 0\\
0 \ar[r] & P \ar[r] & \Sigma P \ar[r] & DP \ar[r] & 0,
}
\end{equation*}
and hence an exact sequence
\begin{equation*}
0 \to \ker \alpha \to V \to \Sigma V \to DV \to 0.
\end{equation*}

Since $W$ is $\sharp$-filtered, $DW$ is $\sharp$-filtered as well, and is torsionless by Proposition \ref{sharp-filtered modules are torsionless}. Therefore, $\ker \alpha$ as a submodule of $DW$ is torsionless as well. However, as explained in Remark \ref{kernel of V to Sigma V}, the kernel of $V \to \Sigma V$ is torsion since it is a submodule of $V_T$, the torsion part of $V$. This happens if and only if the kernel is 0. That is, $V$ is torsionless.
\end{proof}

The conclusion of this lemma can be strengthened.

\begin{lemma} \label{second homology vanishes}
Let $V$ be a finitely generated $\C$-module. If $H_2(V) = 0$, then $V$ is $\sharp$-filtered.
\end{lemma}

\begin{proof}
We already know that $V$ is torsionless from the previous lemma. Now we use induction on $\gd(V)$. If $\gd(V) = 0$, then the conclusion follows from Lemma \ref{modules with generating degree 0}. Otherwise, we have a commutative diagram
\begin{equation*}
\xymatrix{
0 \ar[r] & W \ar[r] \ar[d] & \Sigma W \ar[r] \ar[d] & DW \ar[r] \ar[d] & 0\\
0 \ar[r] & P \ar[r] \ar[d] & \Sigma P \ar[r] \ar[d] & DP \ar[r] \ar[d] & 0\\
0 \ar[r] & V \ar[r] & \Sigma V \ar[r] & DV \ar[r] & 0.
}
\end{equation*}
where $P$ is projective and $W$ is $\sharp$-filtered. By Proposition \ref{sharp-filtered modules are torsionless}, $DW$ is $\sharp$-filtered as well. Moreover, $\gd(DV) < \gd(V)$. Therefore, by induction hypothesis, $DV$ is $\sharp$-filtered. But by Lemma \ref{crucial recursion}, $V$ must be $\sharp$-filtered as well.
\end{proof}

The conclusion of Corollary \ref{kernels of sharp-filtered modules are sharp-filtered} can be strengthened as follows:

\begin{corollary} \label{SES of sharp-filtered modules}
Let $0 \to U \to V \to W \to 0$ be a short exact sequence of finitely generated $\C$-modules. If two terms are $\sharp$-filtered, so is the third one.
\end{corollary}

\begin{proof}
It suffices to show that if $U$ and $V$ are $\sharp$-filtered, so is $W$. The long exact sequence
\begin{equation*}
\ldots \to H_2(U) \to H_2(V) \to H_2(W) \to H_1(U) \to \ldots
\end{equation*}
tells us $H_2(W) = 0$ since $H_1(U) = H_2(V) = 0$. The conclusion follows from Lemma \ref{second homology vanishes}.
\end{proof}

Now we are ready to prove the first main theorem mentioned in Section 1.

\begin{proposition}
Let $V$ be a finitely generated $\C$-module. Then $V$ is $\sharp$-filtered if and only if $H_s (V) = 0$ for some $s \geqslant 1$.
\end{proposition}

\begin{proof}
One direction is trivial. For the other direction, we can assume that $s \geqslant 2$. Take an adaptable projective resolution $P^{\bullet} \to 0$ and let $Z^i$ be the $i$-th cycle. Then we have
\begin{equation*}
0 = H_s(V) = H_1 (Z^{s-1}).
\end{equation*}
Consequently, $Z^{s-1}$ is a $\sharp$-filtered module. Applying the previous corollary to the short exact sequence
\begin{equation*}
0 \to Z^{s-1} \to P^{s-2} \to Z^{s-2} \to 0
\end{equation*}
one deduces that $Z^{s-2}$ is $\sharp$-filtered as well. The conclusion follows from recursion.
\end{proof}

\subsection{Upper Bounds for projective dimensions}

In this subsection we prove Theorem \ref{second main result}. We need a well known result on representation theory of finite groups.

\begin{lemma} \label{pd of group representations}
Let $G$ be a finite group and $\mk$ be a commutative Noetherian ring. Let $V$ be a finitely generated $\mk G$-module. If $\pd _{\mk G} (M)$ is finite, then $\pd _{\mk G} (M) = \pd_{\mk} (M)$.
\end{lemma}

\begin{proof}
The proof uses Eckmann-Shapiro Lemma. For details, see \cite[Theorem 4.3]{Li}.
\end{proof}

If $V$ is a basic $\sharp$-filtered module, then its projective dimension coincides with that of the finitely generated group representation inducing $V$. That is,

\begin{lemma} \label{compare pd}
Let $T$ be a finitely generated $\mk S_i$-module. Then
\begin{equation*}
\pd _{\mk S_i} (T) = \pd _{\tC} (\tC \otimes _{\mk S_i} T).
\end{equation*}
\end{lemma}

\begin{proof}
We claim that
\begin{equation*}
\pd _{\mk S_i} (T) \geqslant \pd _{\tC} (\tC \otimes _{\mk S_i} T).
\end{equation*}
If $\pd _{\mk S_i} (T) = \infty$, the inequality holds. Otherwise, there is a projective resolution $Q^{\bullet} \to T \to 0$ of $\mk S_i$-modules such that $Z^s$ is projective for $s \geqslant \pd _{\mk S_i} (T)$, where $Z^s$ is the $s$-th cycle. Applying the exact functor $\tC \otimes _{\mk S_i} -$ one gets a projective resolution of $\C$-modules
\begin{equation*}
\tC \otimes _{\mk S_i} Q^{\bullet} \to \tC \otimes _{\mk S_i} T \to 0,
\end{equation*}
whose $s$-th cycle $\tC \otimes _{\mk S_i} Z^s$ is a projective $\C$-module for $s \geqslant \pd _{\mk S_i} (T)$. The claim is proved.

Now we show that
\begin{equation*}
\pd _{\mk S_i} (T) \leqslant \pd _{\tC} (\tC \otimes _{\mk S_i} T).
\end{equation*}
If $\pd _{\tC} (\tC \otimes _{\mk S_i} T) = \infty$, the inequality holds. Otherwise, there is a finite projective resolution of $\C$-modules
\begin{equation*}
P^{\bullet} \to \tC \otimes _{\mk S_i} T \to 0
\end{equation*}
such that each term of which is generated in degree $i$. Restricting this resolution to the object $i$ one gets a finite resolution of $\mk S_i$-modules
\begin{equation*}
1_i P^{\bullet} \to 1_i (\tC \otimes _{\mk S_i} T) = T \to 0.
\end{equation*}
The above inequality follows from this observation.
\end{proof}

The following lemma tells us that to consider the projective dimension of a $\sharp$-filtered module, it is enough to consider its basic filtration components.

\begin{lemma} \label{pd of basic components}
Let $V$ be a finitely generated $\sharp$-filtered $\C$-module and suppose that $\pd(V) < \infty$. Let
\begin{equation*}
0 = V^{-1} \subseteq V^0 \subseteq V^1 \subseteq \ldots \subseteq V^n = V
\end{equation*}
be the filtration given in Proposition \ref{structure of sharp-filtered modules}. Then for $0 \leqslant i \leqslant n$, one has
\begin{equation*}
\pd (V^i / V^{i-1}) \leqslant \pd(V).
\end{equation*}
\end{lemma}

\begin{proof}
One uses induction on the length $n$, which is precisely $\gd(V)$. If $n = 0$ or $-\infty$, nothing needs to show. For $n \geqslant 1$, one has a short exact sequence
\begin{equation*}
0 \to U \to V \to W \to 0
\end{equation*}
where $U$ is the submodule generated by $\bigoplus _{i \leqslant n-1} V_i$. It suffices to show that both $\pd (U) \leqslant \pd(V)$ and $\pd(W) \leqslant \pd(V)$ since in that case by induction hypothesis one knows that each basic filtration component of $U$ has projective dimension at most $\pd(U) \leqslant \pd(V)$.

Take two surjections $P^0 \to U$ and $Q^0 \to W$ such that $\gd(P^0) = \gd(U) \leqslant n-1$ and $Q^0$ is generated in degree $n$. We get a commutative diagram
\begin{equation*}
\xymatrix{
0 \ar[r] & U^{(1)} \ar[r] \ar[d] & V^{(1)} \ar[r] \ar[d] & W^{(1)} \ar[r] \ar[d] & 0\\
0 \ar[r] & P^0 \ar[r] \ar[d] & P^0 \oplus Q^0 \ar[r] \ar[d] & Q^0 \ar[r] \ar[d] & 0\\
0 \ar[r] & U \ar[r] & V \ar[r] & W \ar[r] & 0.
}
\end{equation*}
Clearly, $U^{(1)}$, $V^{(1)}$, and $W^{(1)}$ are all $\sharp$-filtered modules (might be 0). Moreover, one still has $\gd(U^{(1)}) \leqslant n-1$ and $\gd(W^{(1)}) = n$ if $W^{(1)}$ is not 0.

Continuing this process, one gets a projective resolution $P^{\bullet} \oplus Q^{\bullet} \to V \to 0$. Since $\pd(V) < \infty$, there exists some $i \in \Z$ such that $V^{(i)}$ is projective. But from the short exact sequence
\begin{equation*}
0 \to U^{(i)} \to V^{(i)} \to W^{(i)} \to 0
\end{equation*}
one sees that both $U^{(i)}$ and $W^{(i)}$ must be projective since $\gd(U^{(i)}) \leqslant n-1$ and $W^{(i)}$ is 0 or generated in degree $n$. This finishes the proof.
\end{proof}

\begin{definition}
The finitistic dimension of $\mk$, denoted by $\fdim \mk$, is defined to be
\begin{equation*}
\sup \{ \pd _{\mk} (T) \mid T \text{ is a finitely generated $\mk$-module and } \pd_{\mk} (T) < \infty \}.
\end{equation*}
\end{definition}

Now we restate and prove Theorem \ref{second main result}.

\begin{theorem}
Let $\mk$ be a commutative Noetherian ring whose finitistic dimension is finite, and let $V$ be a finitely generated $\C$-module with $\gd(V) = n$. Then the projective dimension $\pd (V)$ is finite if and only if for $0 \leqslant i \leqslant n$, one has
\begin{equation*}
V^i / V^{i-1} \cong \tC \otimes _{\mk S_i} (V^i / V^{i-1})_i
\end{equation*}
and
\begin{equation*}
\pd _{\mk S_i} ((V^i / V^{i-1})_i) < \infty,
\end{equation*}
where $V^i$ is the submodule of $V$ generated by $\bigoplus _{j \leqslant i} V_j$, Moreover, in that case
\begin{equation*}
\pd(V) = \max \{ \pd _{\mk S_i} ((V^i / V^{i-1})_i) \} _{i=0}^n = \max \{ \pd _{\mk} ((V^i / V^{i-1})_i) \} _{i=0}^n \leqslant \fdim \mk.
\end{equation*}
\end{theorem}

\begin{proof}
If $\pd(V) < \infty$, then $H_s(V) = 0$ for $s \gg 0$. By Theorem \ref{characterizations of sharp-filtered modules}, $V$ must be a $\sharp$-filtered module. By Proposition \ref{structure of sharp-filtered modules},
\begin{equation*}
V^i / V^{i-1} \cong \tC \otimes _{\mk S_i} (V^i / V^{i-1})_i.
\end{equation*}
By Lemmas \ref{compare pd} and \ref{pd of basic components},
\begin{equation*}
\pd _{\mk S_i} ((V^i / V^{i-1})_i) = \pd _{\tC} (V^i / V^{i-1}) \leqslant \pd _{\tC} (V) < \infty.
\end{equation*}
Conversely, if the structure of $V$ has the given description, then the filtration and Lemma \ref{compare pd} tell us that
\begin{equation*}
\pd _{\tC} (V) \leqslant \max \{ \pd _{\tC} (\tC \otimes _{\mk S_i} (V^i / V^{i-1})_i) \} _{i=0}^n = \max \{ \pd _{\mk S_i} ( (V^i / V^{i-1})_i) \} _{i=0}^n,
\end{equation*}
which is finite.

Now suppose that $\pd _{\tC} (V)$ is finite. Then Lemma \ref{pd of basic components} tells us
\begin{equation*}
\pd _{\tC} (V) \geqslant \max \{ \pd _{\tC} (V^i / V^{i-1}) \} _{i=0}^n = \max \{ \pd _{\mk S_i} ((V^i / V^{i-1})_i) \} _{i=0}^n.
\end{equation*}
Therefore,
\begin{equation*}
\pd _{\tC} (V) = \max \{ \pd _{\mk S_i} ((V^i / V^{i-1})_i) \} _{i=0}^n = \max \{ \pd _{\mk} ( (V^i / V^{i-1})_i) \} _{i=0}^n
\end{equation*}
by Lemma \ref{pd of group representations}. Since all numbers in the last set must be finite, clearly $\pd _{\tC} (V) \leqslant \fdim \mk$ by the definition of finitistic dimensions.
\end{proof}

\begin{remark}
There does exist a finitely generated $\C$-module $V$ whose projective dimension is exactly $\fdim \mk$. Indeed, let $T$ be a $\mk$-module with $\pd_{\mk} (T) = \fdim \mk$, then
\begin{equation*}
\pd _{\mk S_i} (\mk S_i \otimes _{\mk} T) = \pd_{\mk} (T) = \fdim \mk
\end{equation*}
and
\begin{equation*}
\pd _{\tC} (\tC \otimes _{\mk S_i} (\mk S_i \otimes _{\mk} T)) = \pd _{\mk S_i} (\mk S_i \otimes _{\mk} T) = \fdim \mk.
\end{equation*}
\end{remark}

The following corollaries are immediate.

\begin{corollary}
If $\gldim \mk < \infty$, then the projective dimension of a finitely generated $\C$-module $V$ is either $\infty$ or at most $\gldim k$.
\end{corollary}

For instance, if $\mk$ is $\mathbb{Z}$ or the polynomial ring of one variable over a field, then the projective dimension of a finitely generated $\C$-module $V$ can only be 0, 1 or $\infty$.

\begin{corollary}
If $\fdim \mk = 0$, then a finitely generated $\C$-module $V$ has finite projective dimension if and only if $V$ is projective.
\end{corollary}

\begin{remark} \normalfont
Actually, many important classes of rings have finitistic dimension 0. Examples includes semisimple rings, self-injective algebras, finite dimensional local algebras, etc.
\end{remark}

\subsection{Complexes of $\sharp$-filtered modules}

In this subsection we construct a finite complex of $\sharp$-filtered modules for every finitely generated $\C$-module.

\begin{theorem} [\cite{N}, Theorem A] \label{complex of sharp-filtered modules}
Let $\mk$ be a commutative Noetherian ring and let $V$ be a finitely generated $\C$-module. Then there exists a complex
\begin{equation*}
F^{\bullet}: \quad 0 \to V \to F^0 \to F^1 \to \ldots \to F^n \to 0
\end{equation*}
satisfying the following conditions:
\begin{enumerate}
\item each $F^i$ is a $\sharp$-filtered module with $\gd(F^i) \leqslant \gd(V) - i$;
\item $n \leqslant \gd(V)$;
\item the homology in each degree of the complex is a torsion module, including the homology at $V$.
\end{enumerate}
\end{theorem}

\begin{proof}
One may assume that $V$ is nonzero. Denote $V$ by $V^0$ and let $N_0$ be a sufficiently large integer. There is a short exact sequence
\begin{equation*}
0 \to V^0_T \to V^0 \to V^0_F \to 0
\end{equation*}
such that $V^0_T$ is torsion and $V^0_F$ is torsionless, which gives a short exact sequence
\begin{equation*}
0 \to \Sigma_{N_0} V^0_T \to \Sigma_{N_0} V^0 \to \Sigma_{N_0} V^0_F \to 0.
\end{equation*}
Since $N_0$ is sufficiently large, we conclude that the first term in the above sequence is 0. Moreover, $\Sigma_{N_0} V^0_F \cong \Sigma_{N_0} V$ is $\sharp$-filtered by Theorem \ref{modules become sharp-filtered}. Let $F^0 = \Sigma_{N_0} V^0_F$. The natural embeddings
\begin{equation*}
V^0_F \to \Sigma V^0_F \to \Sigma_2 V^0_F \to \ldots \to \Sigma_{N_0} V^0_F = F^0
\end{equation*}
induces a map $\delta_{-1}: V^0 \to F^0$ which is the composite $V^0 \to V^0_F \to F^0$. Let $V^1 = \coker \delta_{-1}$. Repeating the above construction we get a map $V^1 \to F^1$, which induces the map $\delta_0: F^0 \to F^1$. In this way we construct a complex
\begin{equation*}
F^{\bullet}: \quad V \to F^0 \to F^1 \to \ldots \to F^n \to \ldots
\end{equation*}

Note that for $i \geqslant 1$, we have
\begin{equation}
\gd(F^i) \leqslant \gd (V_F^i) \leqslant \gd(V^i) \leqslant \gd(F^{i-1}) - 1
\end{equation}
where the first inequality holds because $F^i = \Sigma _{N_i} V^i_F$, and the last inequality follows from the exact sequence
\begin{equation*}
0 \to V^{i-1}_F \to F^{i-1} = \Sigma_{N_i} V^{i-1}_F \to V^i \to 0
\end{equation*}
and an argument similar to the proof of (3) in Proposition \ref{properties of D}. Therefore, using induction one concludes that
\begin{equation*}
\gd(F^i) \leqslant \gd(F^0) - i \leqslant \gd(V) - i
\end{equation*}
since $\gd(F^0) \leqslant \gd(V)$. This proves (1), which immediately implies (2). Moreover, one has
\begin{equation}
\gd(V^i) \leqslant \gd(V) - i
\end{equation}
for $i \geqslant 1$.

Now we prove (3). Note that the image of the map $\delta_i: F^i \to F^{i+1}$ is precisely $V^{i+1}_F$. The commutative diagram
\begin{equation*}
\xymatrix{
 & 0 \ar[r] & \im \delta_{i-1} = V^i_F \ar[r] \ar[d] & \ker \delta_i \ar[r] \ar[d] & V_T^{i+1} \ar[r] & 0\\
 & & F^i \ar@{=}[r] \ar[d] & F^i \ar[d]\\
0 \ar[r] & V^{i+1}_T \ar[r] & \coker \delta_{i-1} = V^{i+1} \ar[r] & \im \delta_i = V_F^{i+1} \ar[r] & 0
}
\end{equation*}
tells us that the homology at $F^i$ is isomorphic to $V_T^{i+1}$ for $i \geqslant 0$, a torsion module. A similar computation tells us that the homology at $V$ is $V_T$. This finishes the proof.
\end{proof}

\begin{remark} \normalfont
When $\mk$ is a field of characteristic 0, $\sharp$-filtered modules coincide with projective modules, which have been shown to be injective as well. Therefore, the above construction generalizes the construction of finite injective resolutions described in \cite{GL2}.
\end{remark}

\begin{remark} \normalfont
We used the conclusion of Theorem \ref{modules become sharp-filtered} to prove the above theorem. But they are actually equivalent. Indeed, choosing a sufficiently large $N$ and applying $\Sigma_N$ to the complex $F^{\bullet}$, we get a complex $\Sigma_N F^{\bullet}$. Since all homologies in $F^{\bullet}$ are torsion modules, after applying $\Sigma_N$, these torsion modules all vanish. Consequently, $\Sigma_N F^{\bullet}$ is a right resolution of $\Sigma_N V$. By Lemma \ref{sharp-filtered modules are torsionless}, each $\Sigma_N F^i$ is still $\sharp$-filtered. By Corollary \ref{kernels of sharp-filtered modules are sharp-filtered}, $\Sigma_N V$ is a $\sharp$-filtered module, as claimed by Theorem \ref{modules become sharp-filtered}.
\end{remark}

\subsection{Another bound for homological degrees}

In this subsection we use the complex of $\sharp$-filtered modules to obtain another bound for homological degrees. The proof of this result is almost the same as that of \cite[Theorem 5.18]{Li2} via replacing projective modules by $\sharp$-filtered modules. For the convenience of the reader, we still give enough details.

\begin{lemma} \label{hd of torsionless modules}
Let $V$ be a finitely generated torsionless $\C$-module. Then for $s \geqslant 1$,
\begin{equation*}
\hd_s (V) \leqslant 2\gd(V) + s - 1.
\end{equation*}
\end{lemma}

\begin{proof}
By the proof of Theorem \ref{complex of sharp-filtered modules}, there is a short exact sequence
\begin{equation*}
0 \to V \to F \to W \to 0
\end{equation*}
where $F = \Sigma_N V$ is a $\sharp$-filtered module and $\gd(W) \leqslant \gd(V) - 1$. Using the long exact sequence
\begin{equation*}
\ldots \to H_2(W) \to H_1(V) \to H_1(F) = 0 \to H_1(W) \to H_0(V) \to H_0(F) \to H_0(W) \to 0
\end{equation*}
one deduces that $\hd_s(V) = \hd_{s+1}(W)$ for $s \geqslant 1$ and $\hd_1(W) \leqslant \gd(V)$.

By \cite[Theorem A]{CE}, we have
\begin{equation*}
\hd_s(V) = \hd_{s+1} (W) \leqslant \gd(W) + \hd_1(W) + s.
\end{equation*}
for $s \geqslant 1$. Consequently,
\begin{equation*}
\hd_s(V) \leqslant \gd(V) - 1 + \gd(V) + s = 2\gd(V) + s - 1
\end{equation*}
as claimed.
\end{proof}

Now we can prove Theorem \ref{third main result}.

\begin{theorem}
Let $\mk$ be a commutative Noetherian ring and $V$ be a finitely generated $\C$-module. Then for $s \geqslant 1$, we have
\begin{equation*}
\hd_s (V) \leqslant \max \{ \td(V), \, 2\gd(V) - 1 \} + s.
\end{equation*}
\end{theorem}

\begin{proof}
The short exact sequence
\begin{equation*}
0 \to V_T \to V \to V_F \to 0
\end{equation*}
induces a long exact sequence
\begin{equation*}
\ldots \to H_s(V_T) \to H_s(V) \to H_s(V_F) \to \ldots.
\end{equation*}
We deduce that
\begin{equation*}
\hd_s(V) \leqslant \max \{ \hd_s(V_T), \, \hd_s(V_F) \}.
\end{equation*}
Note that
\begin{equation*}
\hd_s(V_T) \leqslant \td(V_T) + s = \td(V) + s,
\end{equation*}
by Corollary \ref{hd of torsion modules} and
\begin{equation*}
\hd_s(V_F) \leqslant 2\gd(V_F) + s - 1 \leqslant 2\gd(V) + s - 1
\end{equation*}
by the previous lemma. The conclusion follows.
\end{proof}

\end{document}